\documentclass{amsart}

\usepackage{graphicx}
\usepackage{fancyhdr}
\usepackage{amsthm}
\usepackage{rotate}
\usepackage{graphicx}
\usepackage[T1]{fontenc}
\usepackage{afterpage}  %
\usepackage{comment}
\newtheorem{theorem}{Theorem}[section]
\newtheorem{lemma}[theorem]{Lemma}
\newtheorem{corollary}{Corollary}[theorem]
\theoremstyle{definition}
\newtheorem{definition}[theorem]{Definition}

\theoremstyle{remark}
\newtheorem{remark}[theorem]{Remark}

\numberwithin{equation}{section}

\DeclareUnicodeCharacter{2212}{-}
\newtheorem{proposition}{Proposition}
\begin{document}
\title{An embedding theorem for Nijenhuis Lie algebras}


\author{Alireza Najafizadeh}
\address{Department of Mathematics, Payame Noor University, P.O.Box 19395-3697 Tehran, Iran}
\email{najafizadeh@pnu.ac.ir}

\author{Chia Zargeh}
\address{Modern College of Business and Science, Muscat, Sultanate of Oman}
\email{Chia.Zargeh@mcbs.edu.om}



\date{}

\maketitle

\section{Abstract}
This paper introduces the Higman-Neumann-Neumann extension (HNN extension; for short) for Nijenhuis Lie algebras and provides an embedding theorem. To this end, we employ the theory of Gr\"obner-Shirshov basis for Lie $\Omega$-algebras in order to find a normal form for our construction. Then we show that every Nijenhuis Lie algebra embeds into its HNN-extension. 

\section{Introduction}
Nijenhuis Lie algebras are a subclass of Lie $\Omega$-algebras. Nijenhuis operator on Lie algebras has been used in the study of integrability of nonlinear evolution equations in \cite{Dorfman}. Recently, the Nijenhuis operator has been studied in the context of various algebraic structures. Asif in \cite{Asif} studied several aspects of Nijuenhuis Lie conformal algebras and provided a classification theory for non-abelian extensions of Nijenhuis Lie conformal algebras. Moreover, Das \cite{Das} investigated the cohomology theory of Nijenhuis Lie algebras. Nijenhuis operator has been studied on other algebraic structures such as Leibniz algebras \cite{Mondal}, $n$-Lie algebras \cite{Liu}, Hom-Lie conformal algebras \cite{Asif2} and Lie bialgebra \cite{Li} with some applications in mathematical physics \cite{Chen}. In this work, we construct the Higman-Neumann-Neumann extension, originally a construction in combinatorial group theory \cite{Higman}, and explore its applications to Nijenhuis Lie algebras. To this end, we employ Gr\"obner-Shirshov basis theory of Lie $\Omega$-algebras that enables us to provide a normal form for a presentation of HNN-extension via structural constants. Then we provide an embedding theorem. The concept of HNN-extension has been spread to various algebraic structures, including Lie (super) algebras and Hom-Lie algebras, Leibniz algebras, Poisson algebras and dialgebras, see \cite{Arenas1}, \cite{Arenas2}, \cite{Klein}, \cite{Ladra1}, \cite{Ladra2}, \cite{Silvestrov} and \cite{Zargeh1}. It has been used to prove embedding theorems and the undecidability of Markov properties   \cite{Najafizadeh}, \cite{Shahryari}, \cite{Wasserman} and \cite{Zargeh2}. The theory of Gr\"obner-Shirshov basis which plays a key role in our construction has been developed in recent decades and has resulted in solving many combinatorial questions and algorithmic problems. In this work, we employ Gr\"obner-Shirshov basis theory for Lie $\Omega$-algebras \cite{Qiu}. There are several comprehensive works introducing the Gröbner-Shirshov basis theory; see, for instance, \cite{Bokut2}  and the references therein. 

The paper is organized as follows. In the first section, we provide an overview of Nijenhuis Lie algebras. In Section $2$, we recall the Composition-Diamond Lemma for Lie $\Omega$-algebras. In the third section, we construct HNN-extension and provide an embedding theorem.

\section{An overview on Nijenhuis Lie algebras}
A Nijenhuis Lie algebra is defined as $(\mathfrak{L}, [\cdot,\cdot], N),$  denoted by $N\mathfrak{L}$ where:
\begin{itemize}
    \item $(\mathfrak{L}, [\cdot,\cdot])$ is a Lie algebra,
    \item $N: \mathfrak{L} \to \mathfrak{L}$ is a linear operator satisfying the {Nijenhuis condition}:
    \[
    [N(x), N(y)] = N\left([N(x), y] + [x, N(y)] - N([x, y])\right) \quad \forall x,y \in \mathfrak{L}.
    \]
\end{itemize}
 \begin{definition}\label{derivation}
        A derivation of a Nijenhuis Lie algebra $N\mathfrak{L}$ is given by $D: N\mathfrak{L} \to N\mathfrak{L}$ that satisfies
        \begin{itemize}
            \item [I.] $D[x,y]=[D(x),y]+[x,D(y)],$
            \item[II.] $D \circ N= N \circ D,$
        \end{itemize}
        for all $x,y \in N\mathfrak{L}.$
    \end{definition}
        \begin{definition}
        A subspace $\mathfrak{S}$ of $N\mathfrak{L}$ is called a Nijenhuis Lie subalgebra if 
        \begin{itemize}
            \item $\mathfrak{S}$ is a Lie subalgebra of $N\mathfrak{L}$, i.e., $[x,y] \in \mathfrak{S},~ \forall x,y \in \mathfrak{S},$
            \item $N(\mathfrak{S}) \subseteq \mathfrak{S},$ $\mathfrak{S}$ is invariant under the Nijenhuis operator. 
        \end{itemize}
    \end{definition}
It is easy to see that an associative Nijenhuis algebra $N\mathfrak{A},$ that is an associative algebra with Nijenhuis operator, is a Lie Nijenhuis algebras $N\mathfrak{L}$, where $$[\mathfrak{a}, \mathfrak{a}']=\mathfrak{a}. \mathfrak{a}' -\mathfrak{a}'. \mathfrak{a},$$ for $\mathfrak{a}, \mathfrak{a}' \in N\mathfrak{A}.$

A Nijenhuis Lie algebra $N\mathfrak{Lie}(X)$ on $X$ with an injective map $i:X \to N\mathfrak{Lie}(X)$ is called a free Nijenhuis Lie algebra on $X$, if for any Nijenuis Lie algebra $N\mathfrak{L}$ and any map $\sigma: X \to N\mathfrak{L},$ there exists a unique Nijenhuis Lie homomorphism $\tilde{\sigma}: N\mathfrak{Lie}(X) \to N\mathfrak{L}$ such that $\tilde{\sigma}i=\sigma.$ The elements of $N\mathfrak{Lie}(X)$ are called Lie Nijenhuis polynomials on $X,$ which we call it Lie $N$-polynomials. The leading term and the leading coefficient of $f$ are denoted by $\bar{f}$ and $lc(f)$, respectively. A polynomial whose leading coefficient is one is called monic. 

\section{Gr\"obner-Shirshov basis for Nijenhuis Lie algebras}
In this section, we recall the Gr\"obner-Shirshov basis and the Composition-Diamond lemma for Lie $\Omega$-algebras according to \cite{Qiu}. We recall some definitions and lemmas from Gr\"obner-Shirshov basis of Lie algebras \cite{Bokut1} that are required in the study of free Nijenhuis Lie algebras. We denote the set of all associative words on $X$ including the empty word 1 by $(X)$ and the set of all non-associative words on $X$ by $[X]$. For any $u \in (X)$, denote $deg(u)$ to be the degree (length) of $u$. So, the lexicographic and degree-lexicographic ordering are defined as:
    \begin{itemize}
        \item[(i)]  $1 >_{lex} u$ for any nonempty word $u$, and if $u = x_iu'$ and $v = x_jv'$, where $x_i, x_j \in X$, then $u >_{lex} v$ if $x_i > x_j$, or $x_i = x_j$ and $u' >_{lex} v'$ by induction.
        \item[(ii)] $u >_{deg-lex} v$ if $deg(u) > deg(v)$, or $deg(u) = deg(v)$ and $u >_{lex} v$. 
    \end{itemize}
    \begin{definition}
        A non-empty associative word $w$ is called a \emph{ Lyndon-Shirshov word} (LS-word for short) on $X$, if $w = uv >_{lex} vu$ for any decomposition of $w = uv$, where $1 \neq u,v \in (X).$ The set of all Lyndon-Shirshov words is denoted by $LS(X).$

    \end{definition}
\begin{definition}
    A non-associative word $[u] \in [X]$, is called a \emph{Lyndon-Shirshov term} (LS-term for short) if: 
\begin{enumerate}
\item $\psi([u]) \in LS(X)$,
\item if $u=[u_1,u_2]$, then both $u_1$ and $u_2$ are LS-terms,
\item in (2) if $u_1=[u_{1,1},u_{1,2}]$, then $\psi([u_{1,2}]) \leq  \psi([u_2])$, 
\end{enumerate}
where $\psi$  is the homomorphism $\psi: [X] \to (X)$ given by $\psi([x])=x$ for every $x \in {X}.$ The set of all Lyndon-Shirshov terms is denoted by $LS[X].$
\end{definition}    
 \begin{remark}
  For any $u \in LS(X)$, there exists a unique Shirshov's bracketing method such that $[u]_{s} \in LS[X]$. Throughout the remainder of this work, we adopt the notation $[u]$ for Shirshov's method, using the original notation only when required.
 \end{remark}
Given a free Lie algebra $\mathfrak{Lie}(X) $, it is known that the set of Lyndon-Shirshov terms on $X$ forms a basis for this. The following lemma summarizes some key properties of Lyndon-Shirshov words and terms.    
 \begin{lemma}[\cite{Bokut2} and references therein]
      \begin{itemize}
          \item [(i)] If $u \in LS(X),$ then $\overline{[u]}=u.$
          \item[(ii)]  Any $[u]\in [X]$ has a representation in $\mathfrak{Lie}(X)$ as $[u]=\sum\alpha_{i}[u_i],$ where $\alpha_{i}\in K$ and $u_i \in LS(X).$
          \item[(iii)] For any $u \in (X)$, there is a unique decomposition $u=u_1\dots u_m,$ where $u_{i} \in LS(X)$ and $u_{j} \leq_{lex} u_{j+1},$ $1\leq  j \leq m-1.$
          \item[(iv)] Let $u=avb,$ where $u,v \in LS(X)$ and $a,b \in (X)$. Then $[u]=[a[vc]b],$ where $b=cd$ for some $c,d \in (X).$ Then $[v]=a[v]b+ \sum \alpha_{i}a_{i}[v]b_{i},$ where $\alpha_{i}\in K$, $a_{i}, b_{i} \in (X)$ and $a_{i}vb_i < _{deg-lex} avb=u.$ It follows that $\overline{[u]_{v}}=u.$
      \end{itemize} 
   \end{lemma}
 \subsection{Nijenhuis Lyndon-Shirhov words and terms}
    Consider $$N(X)=\bigcup_{m=1}^{\infty} \{N(x) \mid x \in X\},$$
    and define $(N;X)_{0}=(X)$ and $[N;X]_{0}=[X]$. Then assume that we have $(N;X)_{n-1}$ and $[N;X]_{n-1}$. Then
    $$(N;X)_{n}=(X \cup N((N;X)_{n-1}),$$
    $$[N;X]_{n}=[X\cup N([N;X]_{n-1})].$$
    The set of Nijenhuis words and terms on $X$ are denoted by $(N;X)=\bigcup_{n=0}^{\infty}(N;X)_{n}$ and $[N;X]=\bigcup_{n=0}^{\infty}[N;X]_{n}$, respectively. If $u\in X \cup N((N;X)),$ then $u$ is called prime. Therefore, for any $u\in (N;X)$, $u$ can be expressed uniquely in the canonical form $u=u_1\dots u_{n}, \quad n\geq 1,$
    where each $u_i$ is prime. The number $n$ is called the breath of $u$ denoted by $bre(u)$. The degree of $u$, denoted by $deg(u),$ is defined to be the total number of all occurrences of all $x\in X $ and the Nijenhus operator $N.$ The \emph{depth} is defined as $dep(u)=min\{n \mid u \in (N;X)_{n}\}.$ Note that $dep([u])=dep((u)).$

 \subsubsection{Ordering}
    The Deg-lex ordering $>_{Deg-lex}$ is defined on the set of associative Nijenhuis words, $(N;X),$ as follows:
    \[ u=u_1\dots u_n >_{Deg-lex} v=v_1\dots v_m \quad \text{if}\quad wt(u)> wt(v)\quad \text{lexicographically},\]
    where $wt(u)=(deg(u),bre(u),u_1,\dots,u_{n})$. Note that in the Nijenhuis Lie algebra, the ordering on the operator is trivial and $X$ is a well-ordered set of generators.

\begin{definition}
    The Lyndon-Shirshov Nijenhuis words and terms are defined by induction on the depth of Nijenhuis words. For $n=0,$ define 
    $LS(N;X)_{0}=LS(X)$ and $LS[N:X]_{0}=LS[X]$ with respect to $>_{lex}$ order on $(X).$ Assume that we have $LS(N;X)_{n-1}$ and $LS[N;X]_{n-1}=\{[u] \mid u \in LS(N;X)_{n-1})\}.$ Then define
    \[LS(N;X)_{n}=LS(X \cup N(LS(N;X)_{n-1}))\]
    with respect to the $>_{lex}$ order.
\end{definition}
Given $u\in X \cup N(LS(N;X)_{n-1}),$ the bracketing on $u$ is defined as $[u]=u$ for $u \in X$ and $[u]=w([v]),$ if $u=w(v).$ The non-associative Lyndon-Shirshov Nijenhuis terms are derived from bracketing of associative LS-Nijenhuis words and defined as 
\[[X \cup N(LS(N;X)_{n-1})]=\{[u]\mid u \in X \cup N(LS(N;X)_{n-1})\}.\]
Note that any order on associative LS-Nijenhuis words induces an order on the non-associative case. In summary, the Lyndon-Shirshov Nijenhuis words and terms are defined as follows, respectively:
\[LS(N;X)= \bigcup_{n=0}^{\infty} LS(N;X)_{n},\]
\[LS[N;X]=\bigcup_{n=0}^{\infty}LS[N;X]_{n}.\]
\begin{remark}
    The set of all Lyndon-Shirshov Nijenhuis terms forms a linear basis of $N\mathfrak{Lie}(X),$ (Lemma 2.7, \cite{Qiu}).
\end{remark}
 \subsubsection{Composition of polynomials}
    In this part, the possible compositions of polynomials in Nijenhuis Lie algebras are recalled. 
    First, we recall certain polynomials in the context of the Nijenhuis Lie algebras that are used in computation of compositions; for more details, see \cite{Bokut3} and \cite{Qiu}. 

    Let consider $N\mathfrak{A}(X)$ be the free associative Nijenhuis algebra on $X$ and $\star \in X.$ A $\star$-$N$-word is any expression in $(N;X\cup \{\star\})$ with only one occurrence of $\star.$ Then the set of $\star$-$N$-words on $X$ is denoted by $(N;X)^{\star}.$ Let $\pi$ be a $\star$-$N$-word and $s\in N\mathfrak{A}(X).$ Then $\pi\mid_{s}=\pi\mid_{\star\mapsto s}$ is called an $s$-Nijenhuis word denoted by $s$-$N$-word. 
    \begin{definition}[\cite{Qiu}]
        Let $\pi \in (N;X)^{\star}$ and $f \in N\mathfrak{Lie}(X) \subseteq N\mathfrak{A}(X)$ be a monic Lie Nijenhuis polynomial. If $\pi\mid_{\bar{f}} \in LS(N;X),$ then
        \[   [\pi \mid_{f}]_{\bar{f}} = \pi\mid_{f} + \sum \alpha_{i} \pi_{i}\mid_{f},\]
        where $\alpha_{i} \in K$ and $\pi_{i}\mid_{\bar{f}}~ <_{Dl} ~\pi\mid_{\bar{f}}.$
    \end{definition}
   
    It has been shown in \cite{Qiu} that given any $f\in N\mathfrak{Lie}(X)$ and $\pi\mid_{\bar{f}}\in LS(N;X)$, then $[\pi\mid_{f}]_{\bar{f}}=\pi\mid_{f} + \sum\alpha_{i} \pi_{i}\mid_{f},$ where $\alpha_{i}\in K$ and $\pi_{i}\mid_{\bar{f}} <_{Dl} \pi\mid_{\bar{f}}.$
    There are two types of compositions:
    \begin{itemize}
        \item [I.] If $w=\bar{f}a=b\bar{g},$ where $a,b \in (N;X)$ with $bre(w)< bre(\bar{f})+bre(\bar{g}),$ then \[ \langle f,g \rangle_{w}= [fa]_{\bar{f}}-[bg]_{\bar{g}}\]
        is called the \emph{intersection composition} of $f$ and $g$ with respect to $w.$
        \item[II.] If $w=\bar{f}=\pi\mid_{\bar{g}}$ then
        \[ \langle f,g\rangle_{w}=f-[\pi\mid_{g}]_{\bar{g}}\]
        is called the inclusion composition of $f$ and $g$ with respect to $w.$
    \end{itemize}
  If $S$ is a monic subset of $N\mathfrak{Lie}(X)$, then the composition $(f,g)_{w}$ is called \emph{trivial modulo} $(S,w)$ if 
  \[\langle f,g \rangle_{w}=\sum \alpha_{i} [\pi_{i}\mid_{s_{i}}]_{\overline{s_{i}}},\]
  where $\alpha_{i}\in K,$ $s_{i}\in S,$ and $ [\pi\mid_{s_{i}}]_{\overline{s_{i}}}$ is a special normal $s_{i}$-word and $ \pi_{i}\mid_{s_{i}} <_{Dl} w.$ 
  \begin{definition}
      A monic set $S$ is called a \emph{Gr\"obner-Shirshov basis} in $N\mathfrak{Lie}(X)$ if any composition $(f,g)_{w}$ of $f,g\in S$ is trivial modulo $(S,w).$
  \end{definition}
    The following is the Composition-Diamond lemmma.
    \begin{theorem}[\cite{Qiu}]\label{CDlemma}
        Let $S \subseteq N\mathfrak{Lie}(X)$ be a non-empty monic set and $Id(S)$ the ideal of $N\mathfrak{Lie}(X)$ generated by $S.$ Then the followings are equivalent:
        \begin{itemize}
            \item [(i)] $S$ is a Gr\"obner-Shirshov basis in $N\mathfrak{Lie}(X).$
            \item[(ii)] $f\in Id(S) \rightarrow \bar{f}=\pi\mid_{\bar{s}} \in LS(N;X)$ for some $s\in S$ and $\pi \in (N;X)^{\star}.$
            \item[(iii)] The set 
            \[Irr(S)=\{[w]\mid w \in LS(N;X),~ w \neq \pi\mid_{\bar{s}}.~ s\in S, ~ \pi\in (N:X)^{\star}\}  \]
            is a linear basis of the Lie Nijenhuis algebra $N\mathfrak{Lie}(X \mid S).$
            \end{itemize}
    \end{theorem}

\section{HNN construction for Nijenhuis Lie algebras}

Let $ N\mathfrak{L}$ be a Nijenhuis Lie algebra and $\mathfrak{S}$ be a subalgebra of  $ N\mathfrak{L}$. Let $D \colon \mathfrak{S}\to N\mathfrak{L}$ be a derivation defined on the subalgebra $\mathfrak{S},$ and $N'$ is an extended Nijenhuis operator such that $N'\mid_{N\mathfrak{L}}=N.$ Then the corresponding HNN-extension is defined as a Nijenhuis Lie algebra presented by:
\begin{equation}\label{HNN-extension}
   N'H_{D}= \Bigl \langle N\mathfrak{L},t \mid [t,a] =D(a), \quad [t,N(a)]=D(N(a)),~ ~ a \in \mathfrak{S} \Bigl \rangle. 
\end{equation}
Here, $t$ is a new symbol that does not belong to $N\mathfrak{L}$. Note that Definition \ref{derivation} implied that $N'(t)=t.$
Assume that $X^{\prime} = X \cup \{t \}$, where $X$ is a well-ordered basis of $N\mathfrak{L}$ and $t> X$. Let denote by $Y$ the basis of $\mathfrak{S}$. We consider the following polynomials:
\begin{itemize}\label{relations}
    \item[I.] $f_{xy}=[x,y] - \sum_{v} \alpha_{xy}^{v} v$ 
    \vspace{0.1in}
    \item[II.] $f_{xy}^{N'}= [x , N(y)] - \sum_{u}  \sum_v n_y^v \alpha_{xv}^u u, $
    \vspace{0.1in}
    \item[III.] $h_{xy}^{N'}= [N(x) , N(y)]-\sum_v \sum_u \sum_w n_x^v n_y^u \alpha_{vu}^w $
    \vspace{0.1in}
    \item[IV.] $g_{a}=[t,a]- \sum_{v} \delta _{a}^{v} v,$
    \vspace{0.1in}
    \item[V.] $g_{a}^{N'}= [t,N(a)]- \sum_{u}  \sum_v n_a^v \delta_{v}^u u $
\end{itemize}
\vspace{0,1in}
where $v$ is an arbitrary element in $X,$
and $x, y\in X$ such that $x>_{\text{lex}}y$ and $a\in Y$. Therefore, we have  
\[N'H_{D}= \Bigl \langle  X \cup \{t\} \mid f_{xy}, f_{xy}^{N'}, h_{xy}^{N'}  \ \ x,y \in X; \  g_{a}, g_{a}^{N'} \ \  a \in Y \Bigr \rangle  , \]
\vspace{0.1in}
 as a presentation of HNN-extension $H_{D}$ (\ref{HNN-extension}) through structural constants. 
 
 Now, consider $S=\{f, f^{N'}, h^{N'}, g, g^{N'}\}$.
 Note that the coefficients used in $f_{xy}^{N'},$ $h_{xy}^{N'}$ and $g_{a}^{N'}$ are defined on the basis of the operator $N'$ that satisfies the Nijenhuis condition. We have some relations between structure constants due to the Jacobian identity, such as
 $\alpha_{xy}=-\alpha_{yx}$ and
\begin{equation}
\sum_m \left( \alpha_{xy}^v \alpha_{vu}^u + \alpha_{yz}^v \alpha_{vx}^u + \alpha_{zx}^v \alpha_{vy}^u \right) = 0. 
   \end{equation}
   In fact, due to the Nijenhuis operator we have $N(x)=\sum_{v\in X} n_{x} v $, then there are relations due to the derivation $D$ as follows:
    \begin{equation}\label{derivationrelation}
        \sum_{v} \alpha_{ab}^{v} \delta_{v}^{u}=\sum_{v}(\delta_{a}^{v} \alpha_{vb}^{u} + \delta_{b}^{v} \alpha_{av}^{u}),
    \end{equation}
    \begin{equation}\label{nijenderirelation}
        \sum_{u,v,k,m} n_a^u n_b^v \alpha_{uv}^k \delta_{k}^m-\sum_{u,v,k,m} n_{a}^u \delta_{v}^u n_b^k \alpha_{uk}^m - \sum_{u,v,k,m} n_b^u \delta_{u}^v n_v^k \alpha_{kv}^m=0,
    \end{equation}
 \begin{equation}
     \sum_{u}\delta_{a}^v n_{v}^u u = \sum_{u} n_{a}^{v} \delta_{v}^u u.
 \end{equation}
 \begin{proposition}
     The set $\{f_{xy}, f_{xy}^{N'}, g_a, g_{a}^{N'}, h_{xy}^{N'}\}$ forms a Gr\"obner-Shirshov basis for the HNN-extension $N'H_{D}.$
 \end{proposition}  
\begin{proof}
    First, we compute all possible compositions. We have $f_{xy}  \wedge f_{yz}$  as follows:
    \begin{align*}
\langle f_{xy},f_{yz}\rangle_{xyz}&=[f_{x y},z]-[x,f_{yz}]\\ 
&=[[x,y],z]-\sum_{v\in X}\alpha_{x y}^v[v,z]-[x,[y,z]]+\sum_{v\in X}\alpha_{ yz}^v [x,v]\\
  & =[y,[z,x]]-\sum_{v\in X}\alpha_{xy}^v[v,z]+\sum_{v\in X}\alpha_{ yz}^v[x,v] \quad \text{ Jacobi identity}\\
&=-[y,[x,z]-\sum_{v\in X}\alpha_{x z}^vv] + \sum_{v \in X} \alpha_{x z} ^{v} ([y,v] -\sum_{u} \alpha_{y v}^{u} u + \sum_{u} \alpha_{y v}^{u} u) \\
&- \sum_{v} \alpha_{x y}^{v} ([v,z] - \sum_{u} \alpha_{vz}^{u} u + \sum_{u} \alpha_{vz}^{u} u)\\
& +  \sum_{v} \alpha_{yz}^{v} ([x,v] - \sum_{u} \alpha_{xz}^{u} u + \sum_{u} \alpha_{xz}^{u} u)\\
&=-[y,f_{xz}]+ \sum_{v \in X} \alpha_{xz} ^{v} f_{yv} - \sum_{v} \alpha_{x y}^{v} f_{vz} +  \sum_{v} \alpha_{yz}^{v} f_{xv}\\
& + \sum_{v} \sum_{u} \alpha_{x z}^{v} \alpha_{y v}^{u} u - \sum_{v} \sum_{u} \alpha_{x y}^{v} \alpha_{vz}^{u} + \sum_{v} \sum_{u} \alpha_{yz}^{v} \alpha_{x z}^{u} u\\ &\equiv -[y,f_{xz}]+ \sum_{v \in X} \alpha_{xz} ^{v} f_{yv} - \sum_{v} \alpha_{x y}^{v} f_{vz} +  \sum_{v} \alpha_{yz}^{v} f_{xv} ~ (mod~S) ~ \text{by (2).}
\end{align*} 
Moreover, $g_a \wedge f_{ab}$ is:
    \begin{align*}
    \langle g_{a}, f_{ab}\rangle_{tab}&=[g_a,b] - [t, f_{ab}]\\
    &=[[t,a] - \sum_{v} \delta_{x}^v v,b ] - [t,[a,b] - \sum_{v} \alpha_{xy}^v v  ]
    \\&=[[t,a], b] - \sum_{v} \delta_{x}^{v} ([v,b] - \sum_{u} \alpha_{vb}^{u} u + \sum_{u} \alpha_{vb}^{u} u)
    \\ &  - [t,[a,b]] + \sum_{v} \alpha_{ab}^{v} ([t,v] - \sum_{u} \delta_{v}^{u} u+  \sum_{u} \delta_{v}^{u} u)
    \\ \text{by Jacobian rule}
    \\&= [[t,b] -  \sum_{v} \delta_{b}^{v} v+\sum_{v} \delta_{b}^{v}v,a] -  \sum_{v} \delta_{a}^{v} q_{vb}
    \\ &- \sum_{v} \sum_{u} \delta_{a}^{v} \alpha_{vb}^{u} u + \sum_{v} \alpha_{ab}^{v} g_v + \sum_{v} \sum_{u} \alpha_{ab}^{v} \delta_{v}^{u} u 
   \\ &=[g_b,a]+\sum_{v} \delta_{b}^{v} ([v,a] - \sum_{u} \alpha_{va}^{u} u + \sum_{u} \alpha_{va}^{u} u )
    \\&-  \sum_{v} \delta_{a}^{v} q_{vb} - \sum_{v} \sum \delta_{a}^{v} \alpha_{vb}^{u} u + \sum_{v} \alpha_{ab}^{v} g_v
    \\& +\sum_{v} \sum \alpha_{ab}^{v} \delta_{v}^{u} u 
   \\ &= [g_b,a]+\sum_{v} \delta_{b}^{v} q_{va} -  \sum_{v} \delta_{a}^{v} q_{vb}+  \sum_{v} \alpha_{ab}^{v} g_v
   \\& + \sum_{v} \sum \delta_{b}^{v} \alpha_{va}^u u 
   - \sum_{v } \sum \delta_{a}^{v} \alpha_{vb}^{u} u +\sum_{v } \sum \alpha_{ab}^{v} \delta_{v}^{u} u\\
&\equiv [g_b,a]+\sum_{v} \delta_{b}^{v} f_{va} -  \sum_{v} \delta_{a}^{v} f_{vb} +  \sum_{v} \alpha_{ab}^{v} g_v ~ (\text{mod}~S).
\end{align*}
The composition $ h_{xy} \wedge h_{yz}$ where $x>_{Deg-lex}y>_{Deg-lex}z$   is an intersection composition. We use the Jacobian rule along with the relations between structure constants to show that it is written as a linear combination of special normal $S$-words, where $S$ is the set of Nijenhuis Lie polynomials $f_{xy}^{N}$-and $h_{xy}^{N}-types$, as follows:
\begin{align*}
    &\langle h_{xy}, h_{yz} \rangle_{N(x)N(y)N(z)}=[[N(x),N(y)]-\sum_{v,u,w} n_x^v n_y^u \alpha_{vu}^w w,N(z)]\\
    &-[N(x),[N(y),N(z)]-\sum_{v,u,w} n_y^v n_z^u \alpha_{vu}^w w] \\
    &= [[N(z),N(x)], N(y)] -\sum_{v,u,w} n_x^v n_y^u \alpha_{vu}^w [w,N(z)]\\
    &+\sum_{v,u,w} n_y^v n_z^u \alpha_{vu}^w [N(x),w]\\
    &=[[N(z),N(x)], N(y)] -\sum_{v,u,w} n_x^v n_y^u \alpha_{vu}^w ([w,N(z)]- \sum_{u}  \sum_m n_z^m \alpha_{wm}^n n\\
    &+ \sum_{u}  \sum_m n_z^m \alpha_{wm}^n n+\sum_{v,u,w} n_y^v n_z^u \alpha_{vu}^w ([N(x),w]- \sum_{u}  \sum_m n_x^m \alpha_{mw}^n n\\
    &+ \sum_{u}  \sum_m n_x^m \alpha_{mw}^n n\\
    &= [[N(z),N(x)]-\sum_{v,u,w} n_z^v n_x^u \alpha_{vu}^w+\sum_{v,u,w} n_z^v n_x^u \alpha_{vu}^w, N(y)]\\
    &-\sum_{v,u,w} n_x^v n_y^u \alpha_{vu}^w f_{wz}^{N}+\sum_{v,u,w} n_y^v n_z^u \alpha_{vu}^w f_{xw}^{N}\\
    &-\sum_{v,u,w} n_x^v n_y^u \alpha_{vu}^w ( \sum_{u}  \sum_m n_z^m \alpha_{wm}^n n)+\sum_{v,u,w} n_y^v n_z^u \alpha_{vu}^w( \sum_{u}  \sum_m n_x^m \alpha_{mw}^n n)\\
    &=[h_{zx}^{N},N(y)]+\sum_{v,u,w} n_z^v n_x^u \alpha_{vu}^w ([w,N(y)]- \sum_{u}  \sum_v n_y^v \alpha_{wv}^u u + \sum_{u}  \sum_v n_y^v \alpha_{wv}^u u)\\
     &-\sum_{v,u,w} n_x^v n_y^u \alpha_{vu}^w f_{wz}^{N}+\sum_{v,u,w} n_y^v n_z^u \alpha_{vu}^w f_{xw}^{N}\\
     &-\sum_{v,u,w} n_x^v n_y^u \alpha_{vu}^w ( \sum_{u}  \sum_m n_z^m \alpha_{wm}^n n)+\sum_{v,u,w} n_y^v n_z^u \alpha_{vu}^w( \sum_{u}  \sum_m n_x^m \alpha_{mw}^n n)\\
     &=[h_{zx}^{N},N(y)]+\sum_{v,u,w} n_z^v n_x^u \alpha_{vu}^w f_{wy}^{N} -\sum_{v,u,w} n_x^v n_y^u \alpha_{vu}^w f_{wz}^{N}\\
     &+\sum_{v,u,w} n_y^v n_z^u \alpha_{vu}^w f_{xw}^{N}\\
     &-\sum_{v,u,w} n_x^v n_y^u \alpha_{vu}^w ( \sum_{u}  \sum_m n_z^m \alpha_{wm}^n n)\\
     &+\sum_{v,u,w} n_y^v n_z^u \alpha_{vu}^w( \sum_{u}  \sum_m n_x^m \alpha_{mw}^n n)\\
     &+\sum_{v,u,w} n_z^v n_x^u \alpha_{vu}^w (\sum_{u}  \sum_v n_y^v \alpha_{wv}^u u)\\
    & \equiv [h_{zx}^{N},N(y)]+\sum_{v,u,w} n_z^v n_x^u \alpha_{vu}^w f_{wy}^{N} -\sum_{v,u,w} n_x^v n_y^u \alpha_{vu}^w f_{wz}^{N}\\
    &+\sum_{v,u,w} n_y^v n_z^u \alpha_{vu}^w f_{xw}^{N}\quad (\text{mod}~S).
\end{align*} 
It is seen that $\overline{[h_{zx}^{N},N(y)]} <_{Deg-lex} [[N(x),N(y)],N(z)],$ which means it is trivial modulo $S$ with respect to $N(x)N(y)N(z).$ The intersection composition $f_{xy}^{N} \wedge f_{yz}^{N}$ where $x > N(y) > z$ is computed as follows:
 \begin{align*}
\langle f_{xy}^{N},f_{yz}^{N}\rangle_{xyz}&= [[x,N(y)]- \sum_{u}  \sum_v n_y^v \alpha_{xv}^u u,z]\\\
&+[x,[N(y),z]- \sum_{u}  \sum_v n_y^v \alpha_{vz}^u u]\\
&=[N(y),[z,x]] - \sum_{u}  \sum_v n_y^v \alpha_{xv}^u [u,z] - \sum_{u}  \sum_v n_y^v \alpha_{vz}^u [x,u]\\
  & =[N(y),[z,x]-\sum_{v} \alpha_{zx}^v v +\sum_{v} \alpha_{zx}^v v]\\
  &- \sum_{u}  \sum_v n_y^v \alpha_{xv}^u([u,z] - \sum_{m}\alpha_{uz}^m m +\sum_{m} \alpha_{uz}^mm)\\
  &- \sum_{u}  \sum_v n_y^v \alpha_{vz}^u ([x,u]-\sum_{m} \alpha_{xu}^m m +\sum_{m} \alpha_{xu}^m m)\\
  &=[N(y),f_{zx}]-\sum_v \alpha_{zx}^v([v,N(y)]- \sum_{m}  \sum_u n_y^u \alpha_{vu}^m m+ \sum_{m}  \sum_u n_y^v \alpha_{vu}^mm)\\
  &- \sum_{u}  \sum_v n_y^v \alpha_{xv}^u f_{uz} - \sum_{u}  \sum_v n_y^v \alpha_{xv}^u(\sum_{m} \alpha_{uz}^mm)\\
  &- \sum_{u}  \sum_v n_y^v \alpha_{vz}^u f_{xu}- \sum_{u}  \sum_v n_y^v \alpha_{vz}^u(\sum_{m} \alpha_{xu}^m m)\\
  &=[N(y),f_{zx}]-\sum_v \alpha_{zx}^v f_{vy}^{N}- \sum_{u}  \sum_v n_y^v \alpha_{xv}^u f_{uz}- \sum_{u}  \sum_v n_y^v \alpha_{vz}^u f_{xu}\\
  &-\sum_v \alpha_{zx}^v ( \sum_{m}  \sum_u n_y^v \alpha_{vy}^m m)- \sum_{u}  \sum_v n_y^v \alpha_{xv}^u(\sum_{m} \alpha_{uz}^mm)\\
  &- \sum_{u}  \sum_v n_y^v \alpha_{vz}^u(\sum_{m} \alpha_{xu}^m m)\\
  &\equiv [N(y),f_{zx}]-\sum_v \alpha_{zx}^v f_{vy}^{N}- \sum_{u}  \sum_v n_y^v \alpha_{xv}^u f_{uz}- \sum_{u}  \sum_v n_y^v \alpha_{vz}^u f_{xu}\\
  &(\text{mod} ~S). \quad (\text{By relations between coefficients due to Jacobi identity})
\end{align*} 
Another intersection composition $g_{a}^{N} \wedge h_{ab}^N$ for $a,b \in \mathfrak{S}$ is computed as follows. Note that, we assume $t> N(a)> N(b).$

\begin{align*}
    \langle g_{a}^{N}, h_{ab}^{N}\rangle_{tN(a)N(b)}&=[[t,N(a)] - \sum_{u}\sum_{v} n_{a}^v\delta_{v}^u u,N(b) ]\\
    &- [t,[N(a),N(b)] - \sum_{m,u,v} n_{a}^v n_{b}^u \alpha_{vu}^m m  ]
    \\&=[[t,N(a)], N(b)]\\
    &- \sum_{u}\sum_{v} n_{a}^v\delta_{v}^u ([u,N(b)] - \sum_{m} \sum_{k} n_{b}^{k}\alpha_{uk}^{m} m + 
    \sum_{m} \sum_{k} n_{b}^{k}\alpha_{uk}^{m} m)
    \\ &  - [t,[N(a),N(b)]] + \sum_{m,u,v} n_{a}^v n_{b}^u \alpha_{vu}^m ([t,m] - \sum_{u} \delta_{m}^{k} k+  \sum_{u} \delta_{m}^{k} k)
    \\ \text{by Jacobian rule}
    \\&= [[t,N(b)] -  \sum_{u} \sum_v n_{b}^v\delta_{N(b)}^{u} u+\sum_u\sum_{v} n_{b}^v \delta_{N(b)}^{u}u,N(a)]\\
    &- \sum_{u}\sum_{v} n_{a}^v\delta_{v}^u f_{ub}^N-\sum_{u}\sum_{v} n_{a}^v\delta_{v}^u (\sum_{m} \sum_{k} n_{b}^{k}\alpha_{uk}^{m} m)\\
     & + \sum_{m,u,v} n_{a}^v n_{b}^u \alpha_{vu}^m g_{m}+ \sum_{m,u,v} n_{a}^v n_{b}^u \alpha_{vu}^m ( \sum_{k} \delta_{m}^{k} k)\\
   &=[g_{b}^{N},N(a)]\\
   &+\sum_u\sum_{v} n_{b}^v \delta_{N(b)}^{u}([u,N(a)]-\sum_{m} \sum_{k} n_{a}^{k} \alpha_{uk}^mm\\
   &+\sum_{m} \sum_{k} n_{a}^{k} \alpha_{uk}^mm)\\
    &- \sum_{u}\sum_{v} n_{a}^v\delta_{v}^u f_{ub}^N-\sum_{u}\sum_{v} n_{a}^v\delta_{v}^u (\sum_{m} \sum_{k} n_{b}^{k}\alpha_{uk}^{m} m)\\
     & + \sum_{m,u,v} n_{a}^v n_{b}^u \alpha_{vu}^m g_{m}+ \sum_{m,u,v} n_{a}^v n_{b}^u \alpha_{vu}^m ( \sum_{k} \delta_{m}^{k} k)\\ 
     &=[g_{b}^{N},N(a)]+\sum_u\sum_{v} n_{b}^v \delta_{N(b)}^{v} f_{ua}^{N} - \sum_{u}\sum_{v} n_{a}^v\delta_{v}^u f_{ub}^N\\
     &+ \sum_{m,u,v} n_{a}^v n_{b}^u \alpha_{vu}^m g_{m}\\
     &+\sum_u\sum_{v} n_{b}^v \delta_{N(b)}^{u}(\sum_{m} \sum_{k} n_{a}^{k} \alpha_{uk}^mm)\\
     &-\sum_{u}\sum_{v} n_{a}^v\delta_{v}^u (\sum_{m} \sum_{k} n_{b}^{k}\alpha_{uk}^{m} m)
     + \sum_{m,u,v} n_{a}^v n_{b}^u \alpha_{vu}^m ( \sum_{u} \delta_{m}^{k} k)\\
     &\equiv [g_{b}^{N},N(a)]+\sum_u\sum_{v} n_{b}^v \delta_{N(b)}^{v} f_{ua}^{N} - \sum_{u}\sum_{v} n_{a}^v\delta_{v}^u f_{ub}^N\\
     &+ \sum_{m,u,v} n_{a}^v n_{b}^u \alpha_{vu}^m g_{m} \quad (\text{mod} S).
 \end{align*}
Based on the above computations, the set $\{f_{xy}, f_{xy}^{N'}, h_{xy}^{N'}, g_{a}, g_{a}^{N'}\}$ is a Gr\"obner-Shirshov basis for the HNN-extension of Nijenhuis Lie algebra $N\mathfrak{Lie}(X).$
\end{proof}
The following theorem provides a normal form the elements of HNN-extension $N'H_{D}.$
\begin{theorem}
    A linear basis for HNN-extension $N'H_{D}$ is given by all the Nijenhuis Lyndon-Shirshov words on $X\cup \{t\}$ which do not contain subwords of the form $xy$ and $N(x)N(y)$ with $x,y \in X$ and $x > y$ and $N(x)>N(y)$, or of the form $ta$ and $tN(a)$ with $a \in  \mathfrak{S}.$
\end{theorem}
\begin{proof}
    The proof is based on the triviality of all possible compositions of Lie and Nijenhuis Lie polynomials and part (iii) of the  Theorem \ref{CDlemma}.
\end{proof}
\begin{corollary}
Every Nijenuis Lie algebra embeds into its HNN-extension. 
\end{corollary}    
\begin{proof}
    All elements of $X$ are words in normal form.
\end{proof}
\begin{corollary}
    Let $x$ be an element in $N\mathfrak{L}$ but not in the subalgebra $\mathfrak{S}$. Then $t$, $x$ and $N(x)$ freely generate a subalgebra generated by them.
\end{corollary}
\begin{proof}
    Let choose the basis $X$ such that $x$ is in $X$ and not in $\mathfrak{S}$. It is seen that all the Nijenhuis Lyndon-Shirshov words on the letters $x$ and $t$ are in normal form, therefore, the subalgebra of $N'H_{D}$ generated by $t,$ $x$ and $N(x)$ is freely generated by them.
\end{proof}

%
%

\end{document}